\documentclass[12pt, reqno,draft]{amsart}
\usepackage{amsthm}
\usepackage{amssymb}
\usepackage{amsrefs}
\usepackage{yhmath}
\usepackage[usenames]{color}
\usepackage{bm}

\newtheorem{thm}{}[section]
\newtheorem{theorem}[thm]{Theorem}
\newtheorem{corollary}[thm]{Corollary}
\newtheorem{lemma}[thm]{Lemma}
\newtheorem{proposition}[thm]{Proposition}
\theoremstyle{definition}
\newtheorem{definition}[thm]{Definition}
\theoremstyle{remark}
\newtheorem{remark}[thm]{Remark}
\newtheorem{question}[thm]{Question}
\newtheorem{example}[thm]{Example}

\numberwithin{equation}{section}
\allowdisplaybreaks

\newcommand{\tl}{\ensuremath{{\bm{t}}}}

\newcommand{\UU}{\ensuremath{U}}
\newcommand{\Ind}{\ensuremath{\mathbf{1}}}

\newcommand{\Fou}{\ensuremath{\mathcal{F}}}
\newcommand{\TT}{\ensuremath{\mathbb{T}}}
\newcommand{\uu}{\ensuremath{\bm u}}
\newcommand{\vv}{\ensuremath{\bm v}}

\newcommand{\NN}{\ensuremath{\mathbb{N}}}
\newcommand{\Nt}{\ensuremath{\mathcal{N}}}
\newcommand{\Mt}{\ensuremath{\mathcal{M}}}
\newcommand{\Ts}{\ensuremath{\mathcal{T}}}
\newcommand{\ZZ}{\ensuremath{\mathbb{Z}}}
\newcommand{\RR}{\ensuremath{\mathbb{R}}}
\newcommand{\FF}{\ensuremath{\mathbb{F}}}
\newcommand{\XX}{\ensuremath{{X}}}
\newcommand{\YY}{\ensuremath{{Y}}}
\newcommand{\xx}{\ensuremath{\bm x}}
\newcommand{\yy}{\ensuremath{\bm y}}
\newcommand{\ee}{\ensuremath{\bm e}}

\newcommand{\ww}{\ensuremath{\bm w}}
\newcommand{\EE}{\ensuremath{\mathcal{E}}}
\newcommand{\BB}{\ensuremath{\mathcal{B}}}
\newcommand{\Id}{\ensuremath{\mathrm{Id}}}
\newcommand{\supp}{\operatorname{supp}}

\begin{document}

\title[Permutative equivalence of squares of unconditional bases]{On the permutative equivalence of squares of unconditional bases}
\author[F. Albiac]{F. Albiac}
\address{Department of Mathematics, Statistics and Computer Sciences, and InaMat\\ Universidad P\'ublica de Navarra\\
Pamplona 31006\\ Spain} 
\email{fernando.albiac@unavarra.es}

\author[J. L. Ansorena]{J. L. Ansorena}
\address{Department of Mathematics and Computer Sciences\\
Universidad de La Rioja\\
Logro\~no 26004\\ Spain} 
\email{joseluis.ansorena@unirioja.es}



\begin{abstract}We prove that if the squares of two unconditional bases are equivalent up to a permutation, then the bases themselves are permutatively equivalent. This settles a twenty year-old question raised by Casazza and Kalton in \cite{CasKal1998}. Solving this problem provides a new paradigm to study the uniqueness of unconditional basis in the general framework of quasi-Banach spaces. Multiple examples are given to illustrate how to put in practice this theoretical scheme. Among the main applications of this principle we obtain the uniqueness of unconditional basis up to permutation of finite sums of quasi-Banach spaces with this property.
\end{abstract}

\thanks{Both authors supported by the Spanish Ministry for Science, Innovation, and Universities, Grant PGC2018-095366-B-I00 for \textit{An\'alisis Vectorial, Multilineal y Approximaci\'on}. The first-named author also acknowledges the support from Spanish Ministry for Economy and Competitivity, Grant MTM2016-76808-P for \textit{Operators, lattices, and structure of Banach spaces}.}

\maketitle

\section{Introduction and background}
\noindent An important long-standing problem in Banach space theory, eventually solved in the negative by Gowers and Maurey in 1997 \cite{GowMau1997}, asked whether any two Banach spaces $X$ and $Y$ such that $X$ is isomorphic to a complemented subspace of $Y$ and such that $Y$ is isomorphic to a complemented subspace of $X$ are isomorphic. This is known, by analogy with a similar result for cardinals in the category of sets, as the \textit{Schr\"oder-Bernstein} problem for Banach spaces.

Pe{\l}czy\'{n}ski had noticed much earlier, back in 1969, that a little extra information about each space, namely being isomorphic to their squares, is all that is needed for the Schr\"oder-Bernstein problem for Banach spaces to have a positive outcome \cite{Pel1969}. This observation, nowadays known as \textit{Pe{\l}czy\'{n}ski's decomposition method}, highlighted the role played by the squares of the spaces, and the question arose whether any two Banach spaces $X$ and $Y$ such that $X^{2}\approx Y^{2}$ are isomorphic. This problem was also settled in the aforementioned article by Gowers and Maurey. 
Indeed, the authors constructed in \cite{GowMau1997} a Banach space $\XX$ with $\XX\approx\XX^3$ but $\XX\not\approx\XX^2$. Then, if we put $\YY=\XX^2$, we have that $\XX$ is isomorphic to a complemented subspace of $\YY$, that $\YY$ is isomorphic to a subspace of $\XX$, that $\XX^2\approx \YY^2$, and that $\XX\not\approx\YY$. So, the pair of spaces $X$ and $Y$ serves as a counterexample for both questions.

The Schr\"oder-Bernstein problem for Banach spaces is a very basic and natural property that arises most of the time when one is trying to show that two Banach (or quasi-Banach) spaces are isomorphic. However, its practical implementation
depends on knowing a priori large classes of spaces when the property holds. And this might be an intractable problem in almost any general setting. 

W\'ojtowicz \cite{Wojto1988} and Wojtaszczyk \cite{Woj1997} discovered independently, with a lapse of 11 years, the following beautiful criterion in the spirit of the Schr\"oder-Bernstein problem to check whether two unconditional bases (in possibly different quasi-Banach spaces) are permutatively equivalent. 

\begin{theorem}[see \cite{Woj1997}*{Proposition 2.11} and \cite{Wojto1988}*{Corollary 1}]\label{thm:SBUB} Let $(\xx_{n})_{n=1}^{\infty}$ and $(\yy_{n})_{n=1}^{\infty}$ be two unconditional bases of quasi-Banach spaces $X$ and $\YY$. Suppose that $(\xx_{n})_{n=1}^{\infty}$ is permutatively equivalent to a subbasis of $(\yy_{n})_{n=1}^{\infty}$ and that $(\yy_{n})_{n=1}^{\infty}$ permutatively is equivalent to a subbasis of $(\xx_{n})_{n=1}^{\infty}$. Then $(\xx_{n})_{n=1}^{\infty}$ and $(\yy_{n})_{n=1}^{\infty}$ are permutatively equivalent. In particular, $X\approx Y$.
\end{theorem}

The validity of the Schr\"oder-Bernstein principle for unconditional bases has a played a crucial role in the development of the subject of uniqueness of unconditional basis in quasi-Banach spaces (see, e.g., \cites{AKL2004, AlbiacLeranoz2008, AlbiacLeranoz2010, AlbiacLeranoz2011, AlbiacLeranoz2011b}). Casazza and Kalton brought this principle to the reader's awareness in \cite{CasKal1998} and used it to give new examples of Banach spaces with a unique unconditional basis up to permutation. The simplifying power of the Schr\"oder-Bernstein principle for unconditional bases would have made life much easier also for all the authors who had previously worked on the problem of uniqueness of unconditional bases up to permutation and who, in order to obtain the same conclusions, had to impose additional properties to the bases in relation to other general techniques such as the decomposition method (see e.g. \cite{BCLT1985}*{Proposition 7.7}). It is indeed remarkable that, although the combinatorial arguments used by Wojtaszczyk to prove Theorem~\ref{thm:SBUB} are somewhat standard, they went unnoticed until close to the 21st century!

The state of art of the Schr\"oder-Bernstein problem for Banach spaces in the pre-Gowers era was described by Casazza in \cite{Casazza1989}. His paper with Kalton \cite{CasKal1998} appeared just one year after Gowers and Maurey disproved the Schr\"oder-Bernstein problem for Banach spaces and Wojtaszczyk's reinterpreted the Schr\"oder-Bernstein principle for unconditional bases. Thus, it is not surprising that the following question was timely raised in \cite{CasKal1998}: 

\begin{question}\label{CKQuestion} (See \cite{CasKal1998}*{Remarks following the proof of Theorem 5.7}) Suppose that $(\xx_{n})_{n=1}^{\infty}$ and $(\yy_{n})_{n=1}^{\infty}$ are two unconditional bases whose squares are permutatively equivalent. Does it follow that $(\xx_{n})_{n=1}^{\infty}$ and $(\yy_{n})_{n=1}^{\infty}$ are permutatively equivalent?
\end{question}

This problem was a driving force for the present investigation and we solve it in the affirmative. In fact we show that the result still holds replacing the assumption on the square of the bases with the weaker assumption that some powers of the bases are permutatively equivalent. We will do that in Section~\ref{sect:SQRTUB}. 

Answering Question~\ref{CKQuestion} in the positive offers a new paradigm to tackle the problem of uniqueness of unconditional basis up to permutation in the general setting of quasi-Banach spaces. The necessary ingredients and preparatory results leading to the main theoretical tool, namely Theorem~\ref{thm:keytechniquebis}, are presented in a self-contained fashion in Section~\ref{sect:UUB}. 

In Sections \ref{Sec:ApplAnti-Euclidean} and \ref{SectStrongAbs} we embark on a comprehensive survey of quasi-Banach spaces with a unique unconditional basis up to permutation which are susceptible to be applied the scheme of Section~\ref{sect:UUB}. 

In Section~\ref{DirectSumsUnc} we further exploit the usefulness of Theorem~\ref{thm:keytechniquebis} to show that the property of uniqueness of unconditional bases is preserved when we take finite direct sums of a wide class of quasi-Banach spaces with this property. When combined with the spaces from Sections \ref{Sec:ApplAnti-Euclidean} and \ref{SectStrongAbs} we obtain a myriad of new examples of spaces with uniqueness of unconditional basis up to permutation. 

We use standard terminology and notation in Banach space theory as can be found, e.g., in \cites{AlbiacKalton2016}. Most of our results, however, will be established in the general setting of quasi-Banach spaces; the unfamiliar reader will find general information about quasi-Banach spaces in \cite{KPR1984}. We next gather the notation that it is more heavily used. In keeping with current usage we will write $c_{00}(J)$ for the set of all $(a_j)_{j\in J}\in \FF^J$ such that $|\{j\in J \colon a_j\not=0\}|<\infty$, where $\FF$ could be the real or complex scalar field. Given $s\in\NN$ we put $\NN[s]=\{1, \dots, s\}.$ Given a quasi-Banach space $\XX$ and $s\in\NN$ we denote by $\kappa[s,\XX]$ the smallest constant $C$ such that
\[
\left\Vert \sum_{j=1}^s f_j\right\Vert\le C \left(\sum_{j=1}^s \Vert f_j\Vert\right),\quad f_j\in\XX.
\]
The closed linear span of a subset  $V$ of $\XX$ will be denoted by $[V]$. A countable family $\BB=(\xx_n)_{n \in \Nt}$ in $\XX$ is an \textit{unconditional basic sequence} if for every $f\in[\xx_n \colon n \in \Nt]$ there is a unique family $(a_n)_{n \in \Nt}$ in $\FF$ such that the series $\sum_{n \in \Nt} a_n \, \xx_n$ converges unconditionally to $f$.
If $\BB$ is an unconditional basic sequence, there is a constant $K\ge 1$ such that
\[
\left\Vert \sum_{n \in \Nt} a_n \, \xx_n\right\Vert \le K \left\Vert \sum_{n \in \Nt} b_n \, \xx_n\right\Vert
\]
for any finitely non-zero sequence of scalars $(a_n)_{n\in \mathcal N}$ with $|a_n|\le|b_n|$ for all $n\in\Nt$ (see \cite{AABW2019}*{Theorem 1.10}). If this inequality  is satisfied for a given $K$ we say that $\BB$ is $K$-unconditional. If we additionally have $[\xx_n \colon n \in \Nt]=\XX$ then $\BB$ is an \textit{unconditional basis} of $\XX$. If $\BB$ is an unconditional basis of $\XX$, then the map
\[
\Fou\colon\XX\to \FF^{\Nt},\quad f=\sum_{n \in \Nt} a_n\, \xx_n \mapsto (\xx_n^*(f))_{n \in \Nt} = (a_n)_{n \in \Nt}
\]
will be called the \textit{coefficient transform} with respect to $\BB$, and the functionals $(\xx_n^*)_{n \in \Nt}$ the \textit{coordinate functionals} of $\BB$.

Given a countable set $\Nt$, we write $\EE_{\Nt}:=(\ee_n)_{n\in \Nt}$ for the canonical unit vector system of $\FF^{\Nt}$, i.e., $\ee_n=(\delta_{n,m})_{m\in \Nt}$ for each $n\in \Nt$, where $\delta_{n,m}=1$ if $n=m$ and $\delta_{n,m}=0$ otherwise. A \textit{sequence space} will be a quasi-Banach space $\XX\subseteq\FF^{\Nt}$ for which $\EE_{\Nt}$ is a normalized $1$-unconditional basis.

The \textit{Banach envelope} of a quasi-Banach space $\XX$ consists of a Banach space $\widehat{\XX}$ together with a linear contraction $J_\XX\colon\XX \to \widehat{\XX}$ satisfying the following universal property: for every Banach space $\YY$ and every linear contraction $T\colon\XX \to\YY$ there is a unique linear contraction $\widehat{T}\colon \widehat{\XX}\to \YY$ such that $\widehat{T}\circ J_\XX=T$. We say that a Banach space $\YY$ is the Banach envelope of $\XX$ via the map $J\colon\XX\to\YY$ if the associated map $\widehat{J}\colon\widehat{\XX}\to\YY$ is an isomorphism. 

Other more specific terminology will be introduced in context when needed.

\section{Permutative equivalence of powers of unconditional bases}\label{sect:SQRTUB}
\noindent Suppose that $\BB_x=(\xx_n)_{n \in \Nt}$ and $\BB_u=(\uu_n)_{n \in \Nt}$ are (countable) families of vectors in quasi-Banach spaces $X$, $Y$, respectively. We say that $\BB_x=(\xx_n)_{n \in \Nt}$ $C$-\textit{dominates} $\BB_u=(\uu_n)_{n \in \Nt}$ if there is a linear map $T$ from the closed subspace of $\XX$ spanned by $\BB_x$ into $\YY$ with $T(\xx_n)=\uu_n$ for all $n \in \Nt$ such that $\Vert T\Vert\le C$. If $T$ is an isomorphic embedding, $\BB_x$ and $\BB_u$ are said to be \textit{equivalent}. We say that $\BB_x$ is \textit{permutatively equivalent} to a family $\BB_y=(\yy_m)_{n\in \Mt}$ in $Y$, and we write $\BB_x\sim\BB_y$, if there is a bijection $\pi\colon \Nt\to \Mt$ such that $\BB_x$ and $(\yy_{\pi(n)})_{n \in \Nt}$ are equivalent. A \textit{subbasis} of an unconditional basis $\BB_x=(\xx_n)_{n \in \Nt}$ is a family $(\xx_n)_{n\in \Mt}$ for some subset $\Mt$ of $\Nt$.

Let $(\XX_i)_{i\in F}$ be a finite collection of (possibly repeated) quasi-Banach spaces. The Cartesian product $\bigoplus_{i\in F}\XX_i$ equipped with the quasi-norm
\[
\left\Vert (\xx_i)_{i\in F}\right\Vert=\sup_{i\in F} \Vert \xx_i\Vert,\quad \xx_i\in\XX_i
\]
is a quasi-Banach space. Suppose that $\BB_i=(\xx_{i,n})_{j\in \Nt_i}$ is an unconditional basis of $\XX_i$ for each $i\in F$. Set 
\begin{equation}\label{defN}
\Nt= \bigcup_{i\in F} \{i\} \times \Nt_i.\end{equation}
Then the countable sequence
$
\bigoplus_{i\in F} \BB_i :=(\xx_{i,n})_{(i,n)\in \Nt}
$
given by $\xx_{i,n} =(\xx_{i,n,j})_{j\in F}$, where
\[
\xx_{i,n,j}=\begin{cases} \xx_{i,n}& \text{ if }i=j, \\ 0 & \text{ otherwise,}
\end{cases}
\]
is an unconditional basis of $\bigoplus_{i\in F} \XX_i$. If $F=\NN[s]$ and $\XX_i=\XX$ for all $i\in F$, the resulting direct sum is called the \textit{$s$-fold product} of $\XX$ and we simply write $\XX^s=\bigoplus_{i\in F} \XX_i$. Similarly, if $\BB_i=\BB$ for all $i\in F=\NN[s]$, we put $\BB^s=\bigoplus_{i\in F} \BB_i$ and say that $\BB^s$ is the $s$-fold product of $\BB$. We will refer to the $2$-fold product of a basis as to the \textit{square} of that basis. We start with an elementary lemma.

\begin{lemma}\label{lem:1} Let $\BB=(\xx_n)_{n\in \Nt}$ be an unconditional basis of a quasi-Banach space $\XX$. For a given $s\in\NN$, consider the $s$-fold product $\BB^s=(\xx_{i,n})_{(i,n)\in \NN[s]\times \Nt}$. Then, for any function $\alpha\colon \Nt\to\NN[s]$, the basic sequence $(\xx_{\alpha(n),n})_{n\in \Nt}$ (which is permutatively equivalent to a subbasis of $\BB^s$) is equivalent to $\BB$.
\end{lemma}

\begin{proof} Suppose that $\BB$ is $K$-unconditional.
If we put $\Nt_i=\alpha^{-1}(i)$ for $i\in\NN[s]$ then
\[
\left\Vert \sum_{n\in \Nt} a_n \, \xx_{\alpha(n),n}\right\Vert=\sup_{i\in\NN[s]} \left\Vert \sum_{n\in \Nt_i} a_n \, \xx_n\right\Vert,
\]
for all $(a_n)_{n=1}^\infty\in c_{00}$. Hence,
\[
\frac{1}{\kappa[s,\XX]}\left\Vert \sum_{n\in \Nt} a_n \, \xx_{n}\right\Vert\le \left\Vert \sum_{n\in \Nt} a_n \, \xx_{\alpha(n),n}\right\Vert \le
K\left\Vert \sum_{n\in \Nt} a_n \, \xx_{n}\right\Vert.\qedhere
\]
\end{proof}

The following version of the Hall-K\"onig Lemma (also known as \textit{Marriage Lemma}) for infinite families of finite sets is essential in the proof of Theorem~\ref{thm:SQRTUB}.

\begin{theorem}[see \cite{Hall1948}*{Theorem 1}]\label{thm:HKL}Let $\Nt$ be a set and $(\Nt_i)_{i\in I}$ be a family of finite subsets of $\Nt$. Suppose that
\[
|F|\le \left| \bigcup_{i\in F} \Nt_i\right|
\]
for every $F\subseteq I$ finite. Then there is a one-to-one map $\phi\colon I\to \Nt$ with $\phi(i)\in \Nt_i$ for every $i\in I$.
\end{theorem}

\begin{theorem}\label{lem:SQRTSubases}Let $\BB_x$ and $\BB_y$ be two unconditional bases of quasi-Banach spaces $\XX$ and $\YY$ respectively. Suppose that $\BB_x^s$ is permutatively equivalent to a subbasis of $\BB_y^s$ for some $s\ge 2$. Then $\BB_x$ is permutatively equivalent to a subbasis of $\BB_y$.
\end{theorem}

\begin{proof} Put $\BB_x=(\xx_{n})_{n\in \Nt}$, $\BB_y=(\yy_{n})_{n\in \Mt}$, $\BB_x^s=(\xx_{i,n})_{(i,n)\in \NN[s]\times \Nt}$ and $\BB_y^s=(\yy_{i,n})_{(i,n)\in \NN[s]\times \Mt}$. By hypothesis there is a one-to-one map
\[
\pi=(\pi_1,\pi_2)\colon \NN[s]\times \Nt \to \NN[s]\times \Mt
\]
such that the unconditional bases $\BB_x^s$ and
$
(\yy_{\pi(i,n)})_{(i,n)\in \NN[s]\times \Nt}
$
are equivalent. For $n\in \Nt$ set $\Mt_n=\{\pi_2(i,n) \colon i\in\NN[s]\}$. If $F$ is a finite subset of $\Nt$ we have
\[
\pi( \NN[s] \times F)\subseteq \NN[s]\times \bigcup_{n\in F} \Mt_n, 
\]
and since $\pi$ is one-to-one,
\[
s \, |F|\le s \left|\bigcup_{n\in F} \Mt_n\right|.
\]
Hence, $|F|\le |\cup_{n\in F} \Mt_n|$. We also have $|\Mt_n|\le s$ for all $n\in \Nt$. Therefore, by Theorem~\ref{thm:HKL}, there exist a one-to-one map $\phi\colon \Nt \to \Mt$, a map $\alpha\colon \Nt\to\NN[s]$, and a map $\beta\colon \Mt\to\NN[s]$ such that
\[
\pi(\alpha(n),n)=(\beta(n),\phi(n)), \quad n\in \Nt,
\]
from where it follows that the unconditional basic sequences $\BB_x'=(\xx_{\alpha(n),n})_{n\in \Nt}$ and $\BB_y'=(\yy_{\beta(n),\phi(n)})_{n\in \Nt}$ are equivalent. Since, on the other hand, by Lemma~\ref{lem:1}, $\BB_x'$ is equivalent to $\BB$ and $\BB_y'$ is permutatively equivalent to $(\yy_m)_{m\in \Mt'}$, where $\Mt'=\phi(\Nt)$, we are done.
\end{proof}

\begin{theorem}\label{thm:SQRTUB} Let $\BB_x$ and $\BB_y$ be two unconditional bases of quasi-Banach spaces $\XX$ and $\YY$. Suppose that $\BB_x^s\sim\BB_y^s$ for some $s\ge 2$. Then $\BB_x\sim \BB_y$.
\end{theorem}

\begin{proof} Apylying Theorem~\ref{lem:SQRTSubases} yields that $\BB_x$ is permutatively equivalent to a subbabis of $\BB_y$, and switching the roles of the basis also the other way around. Using Theorem~\ref{thm:SBUB} closes the proof.
\end{proof}

\begin{corollary}Let $\BB$ be an uconditional basis of a quasi-Banach space. Suppose that $\BB^t$ is permutatively equivalent to a subbasis of $\BB^s$ for some $t>s\ge 1$. Then $\BB^2\sim \BB$.
\end{corollary}
\begin{proof}Since $t\ge s+1$, $\BB^{s+1}$ is permutatively equivalent to a subbasis of $\BB^s$. By induction we deduce that $\BB^{u+1}$ is permutatively equivalent to a subbasis of $\BB^u$ for every $u\ge s$, and so by transitivity, $\BB^u$ is permutatively equivalent to a subbasis of $\BB^s$ for every $u\ge s$. In particular, $\BB^{2s}$ is is permutatively equivalent to a subbasis of $\BB^s$. Therefore, by Theorem~\ref{lem:SQRTSubases}, $\BB^{2}$ is is permutatively equivalent to a subbasis of $\BB$. Since $\BB$ is permutatively equivalent to a subbasis of $\BB^2$, applying Theorem~\ref{thm:SBUB} we are done.
\end{proof}

\section{A new theoretical approach to the uniqueness of unconditional basis in quasi-Banach spaces}\label{sect:UUB}

\noindent From a structural point of view, it is useful to know if a given space has an unconditional basis and, if the answer is yes, whether this is the unique unconditional basis of the space. Recall that a quasi-Banach space $\XX$ with an unconditional basis $\BB$ is said to have a \textit{unique unconditional basis}, if every semi-normalized unconditional basis of $\XX$ is equivalent to $\BB$. For convenience, from now on all bases will be assumed to be semi-normalized. 
Note that, if $\BB=(\xx_n)_{n \in \Nt}$ is a semi-normalized unconditional basis then it is equivalent to the normalized basis $(\xx_n/\Vert \xx_n\Vert)_{n \in \Nt}$. 

For a Banach space with a symmetric basis it is rather unusual to have a unique unconditional basis. It is well-known that $\ell_{2}$ has a unique unconditional basis \cite{KotheToeplitz1934}, and a classical result of Lindenstrauss and Pe{\l }czy{\'n}ski \cite{LinPel1968} asserts that $\ell_{1}$ and $c_{0}$ also have a unique unconditional basis. Lindenstrauss and Zippin \cite{LinZip1969} completed the picture by showing that those three are the only Banach spaces in which all unconditional bases are equivalent. 

Once we have determined that a Banach space does not have a symmetric basis (a task that can be far from trivial) we must rethink the problem of uniqueness of unconditional basis. In fact, an unconditional non-symmetric basis admits a continuum of nonequivalent permutations (cf.\ \cite{Hennefeld1973}*{Theorem 2.1}). Hence for Banach spaces without symmetric bases it is more natural to consider instead the question of uniqueness of unconditional bases up to (equivalence and) a permutation, (UTAP) for short. We say that $\XX$ has a (UTAP) unconditional basis $\BB$ if every unconditional basis in $\XX$ is permutatively equivalent to $\BB$. The first movers in this direction were Edelstein and Wojtaszczyk, who proved that finite direct sums of $c_0$, $\ell_1$ and $\ell_2$ have a (UTAP) unconditional basis \cite{EdelWoj1976}. Bourgain et al.\ embarked on a comprehensive study aimed at classifying those Banach spaces with unique unconditional basis up to permutation, that culminated in 1985 with their \textit{Memoir} \cite{BCLT1985}. They showed that the spaces $c_{0}(\ell_{1})$, $c_{0}(\ell_{2})$, $\ell_{1}(c_{0})$, $\ell_{1}(\ell_{2})$ and their complemented subspaces with unconditional basis all have a (UTAP) unconditional basis, while $\ell_{2}(\ell_{1})$ and $\ell_{2}(c_{0})$ do not. However, the hopes of attaining a satisfactory classification were shattered when they found a nonclassical Banach space, namely the $2$-convexification $\Ts^{(2)}$ of Tsirelson's space having a (UTAP) unconditional basis. Their work also left many open questions, most of which remain unsolved as of today. 

On the other hand, in the context of quasi-Banach spaces that are not Banach spaces, the uniqueness of unconditional basis seems to be the norm rather than an exception. For instance, it was shown in \cite{Kalton1977} that a wide class of nonlocally convex Orlicz sequence spaces, including the $\ell_{p}$ spaces for $0<p<1$, have a unique uncoditional basis. The same is true in nonlocally convex Lorentz sequence spaces (\cites{KLW1990, AlbiacLeranoz2008}) and (UTAP) in the Hardy spaces $H_{p}(\TT)$ for $0<p<1$ (\cite{Woj1997}).

This section is geared towards Theorem~\ref{thm:keytechniquebis}, which tells us that, under three straightforwardly verified conditions regarding a space and a basis, the unconditional bases of a space are all permutatively equivalent. The techniques used in the proof of this theorem are a development of the methods introduced by Casazza and Kalton in \cites{CasKal1998, CasKal1999} to investigate the problem of uniqueness of unconditional basis in a class of Banach lattices that they called \textit{anti-Euclidean}. The subtle but crucial role played by the lattice structure of the space in the proof of Theorem~\ref{thm:keytechniquebis} has to be seen in that it will permit to simplify the untangled way in which the vectors of one basis can be written in terms of the other. These techniques have been extended to the nonlocally convex setting and efficiently used in the literature to establish the uniqueness of unconditional basis up to permutation of the spaces $\ell_{p}(\ell_{q})=(\ell_{q}\oplus \ell_{p}\oplus\dots\oplus \ell_{p}\dots)_{q}$ for $p\in (0,1]\cup \{\infty\}$ and $q\in (0,1]\cup \{2,\infty\}$ (see \cites{AKL2004,AlbiacLeranoz2008, AlbiacLeranoz2010, AlbiacLeranoz2011, AlbiacLeranoz2011b}), with the convention that $\ell_{\infty}$ here means $c_{0}$.

Before moving on, recall that an unconditional basic sequence $\BB_u=(\uu_{m})_{m\in \Mt}$ in a quasi-Banach space $\XX$ is said to be \textit{complemented} if its closed linear span $\UU= [\BB_u]$ is a complemented subspace of $\XX$, i.e., there is a bounded linear map $P\colon\XX\to\UU$ with $P|_\UU=\Id_\UU$. Notice the unconditional basic sequence $\BB_u=(\uu_m)_{m\in \Mt}$ is complemented in $\XX$ if and only if there exists a sequence $(\uu_m^*)_{m\in \Mt}$ in $\XX^*$ such that $\uu_m^*(\uu_n)=\delta_{m,n}$ for every $(m,n)\in \Mt^2$ and a there is a linear bounded map $P_u\colon\XX\to \XX$ given by
\begin{equation}\label{eq:projCUBS}
P_u(f)=\sum_{m\in \Mt} \uu_m^*(f) \, \uu_m, \quad f\in\XX.
\end{equation}
We will refer to $(\uu_m^*)_{m\in \Mt}$ as a sequence of \textit{projecting functionals} for $\BB_u$. 
A family $\BB_u=(\uu_m)_{m\in \Mt}$ in $\XX$ with mutually disjoint supports with respect to a given unconditional basis $\BB$ is an unconditional basic sequence. In the case when, moreover, $\supp(\uu_m)$ is finite for every $m\in \Mt$ we say that $\BB_u$ is a block basic sequence (with respect to $\BB$). We say that the block basic sequence $\BB_u$ is \textit{well complemented} (with respect to $\BB$) if we can choose a sequence of projecting functionals $\BB_{u}^*=(\uu_m^*)_{m\in \Mt}$ with $\supp(\uu_m^*)\subseteq \supp(\uu_m)$ for all $m\in \Mt$. In this case, $\BB_{u}^*$ is called a sequence of \textit{good projecting functionals} for $\BB_{u}$. 

The following definition identifies and gives relief to an unstated feature shared by some unconditional bases. Examples of such bases can be found, e.g., in \cites{Kalton1977, CasKal1998, AlbiacLeranoz2008}, where the property naturally arises in connection with the problem of uniqueness of unconditional basis.

\begin{definition}An unconditional basis $\BB=(\xx_n)_{n\in \Nt}$ of a quasi-Banach space will be said to be \textit{universal for well complemented block basic sequences} if for every semi-normalized well complemented block basic sequence $\BB_u=(\uu_m)_{m\in \Mt}$ of $\BB$ there is a map $\pi\colon \Mt\to \Nt$ such that $\pi(m)\in\supp(\uu_n)$ for every $m\in \Mt$, and $\BB_u$ is equivalent to the rearranged subbasis $(\xx_{\pi(m)})_{m\in \Mt}$ of $\BB$.
\end{definition}

The ideas in the following definition and proposition are implicit in \cite{Kalton1977}.

\begin{definition} An unconditional basis $\BB=(\xx_n)_{n\in \Nt}$ of a quasi-Banach space $\XX$
will be said to have the \textit{peaking property} if every semi-normalized well complemented block basic sequence $\BB_u=(\uu_m)_{m\in \Mt}$ with respect to $\BB$ satisfies
\begin{equation}\label{eq:gh}
\inf_{m\in \Mt} \sup_{n\in \Nt} |\uu_m^*(\xx_n)| \, |\xx_n^*(\uu_m)|>0
\end{equation}
for some sequence $(\uu_m^*)_{m\in \Mt}$ of good projecting functionals for $\BB_u$.
\end{definition}

\begin{proposition}\label{prop:k2one} Suppose $\BB=(\xx_n)_{n \in \Nt}$ is an unconditional basis of a quasi-Banach space $X$. If $\BB$ has the peaking property then it is universal for well complemented block basic sequences.
\end{proposition}

\begin{proof} Let $\BB_u=(\uu_m)_{m\in \Mt}$ be a semi-normalized well complemented block basic sequence and $\BB_u^*=(\uu_m^*)_{m\in \Mt}$ be a sequence of good projecting functionals for $\BB_u$ such that \eqref{eq:gh} holds. There is $\pi\colon \Mt\to \Nt$ one-to-one with
\[
\inf_{m\in \Mt} |\xx_{\pi(m)}^*(\uu_m)| \, |\xx_{\pi(m)}^*(\uu_m)|>0.
\]
For $m\in \Mt$ let us put
\[
\lambda_m= \xx_{\pi(m)}^*(\uu_m),\quad \mu_m=\xx_{\pi(m)}(\uu_m^*),
\]
and set
\[ 
\vv_m=\lambda_m \, \xx_{\pi(m)},\quad \vv_m^*= \mu_m \, \xx^*_{\pi(m)}.
\]
By \cite{AlbiacAnsorena2020}*{Lemma 3.1}, $\BB_v=(\vv_m)_{m \in \Mt}$ is equivalent to $\BB_u$. In particular, $\BB_v$ is semi-normalized so that $\inf_m \lambda_m>0$ and $\sup_m \lambda_m<\infty$. It follows that $\BB_v$ is equivalent to $(\xx_{\pi(m)})_{m \in \Mt}$.
\end{proof}

The last ingredient in the deconstruction process we are carrying out is the following feature about the lattice structure of a quasi-Banach space. 
\begin{definition}A quasi-Banach space (respectively, a quasi-Banach lattice) $\XX$ is said to be \textit{sufficiently Euclidean} if $\ell_2$ is crudely finitely representable in $\XX$ as a complemented subspace (respectively, complemented sublattice), i.e., there is a positive constant $C$ such that for every $n\in\NN$ there are bounded linear maps (respectively, lattice homomorphisms) $I_n \colon\ell_2^n \to \XX$ and $P_n\colon \XX \to \ell_2^n$ with $P_n\circ I_n =\Id_{\ell_2^n}$ and $ \Vert I_n\Vert \, \Vert P_n \Vert\le C$. We say that $\XX$ is \textit{anti-Euclidean} (resp.\ \textit{lattice anti-Euclidean}) if it is not sufficiently Euclidean.
\end{definition}

Any (semi-normalized) unconditional basis of a quasi-Banach space $\XX$ is equivalent to the unit vector system of a sequence space and so it induces a lattice structure on $\XX$. In general, we will say that an unconditional basis has a property about lattices if its associated sequence space has it. And the other way around, i.e., we will say that a sequence space enjoys a certain property relevant to bases if its unit vector system does.

A quasi-Banach lattice $\XX$ is said to be \textit{L-convex} if there is $\varepsilon>0$ so that whenever $f$ and $(f_i)_{i=1}^k$ in $\XX$ satisfy $0\le f_i\le f$ for every $i=1$, \dots, $k$, and $(1-\varepsilon)kf\ge \sum_{i=1}^k f_i$ we have $\varepsilon \Vert f \Vert \le \max_{1\le i \le k} \Vert f_i\Vert$. Kalton \cite{Kalton1984b} showed that a quasi-Banach lattice is $L$-convex if and only if it is $p$-convex for some $p>0$. So, most quasi-Banach lattices (and unconditional bases) ocurring naturally in analysis are L-convex.

The space $\ell_1$ is the simplest and most important example of anti-Euclidean space (see e.g. \cite{AlbiacAnsorena2020}*{Comments previous to Remark 2.9}). So, it is helpful to be able to count on conditions that guarantee that the Banach envelope of a given quasi-Banach space is $\ell_1$.

\begin{lemma}[see \cite{AlbiacAnsorena2020}*{Proposition 2.10}]\label{lem:BEl1}Suppose $\XX$ is a quasi-Banach space with an unconditional basis $\BB$ that dominates the unit vector basis of $\ell_1$. Then the Banach envelope of $\XX$ is $\ell_1$ via the coefficient transform.
\end{lemma}

The following lemma is useful when dealing with unconditional bases that dominate the canonical basis of $\ell_1$. 

Given an unconditional basis $\BB=(\xx_n)_{n \in \Nt}$ with coordinate functionals $(\xx_n^*)_{n \in \Nt}$
and $A\subseteq \Nt$ finite we will put
\[
\Ind_A[\BB]=\sum_{n\in A} \xx_n\quad \text{ and }\quad \Ind_A^*[\BB]=\sum_{n\in A} \xx_n^*.
\]
If $\BB$ is clear from context we simply write $\Ind_A=\Ind_A[\BB]$ and $\Ind_A^*=\Ind_A^*[\BB]$.
\begin{lemma}[cf. \cite{AKL2004}*{Lemma 4.1}]\label{lem:k2two} Let $\BB=(\xx_n)_{n \in \Nt}$ be an unconditional basis of a quasi-Banach space $\XX$. Suppose that $\BB$ dominates the canonical basis of $\ell_1$. Then every semi-normalized well complemented block basic sequence of $\XX$ with respect to $\BB$ is equivalent to a well complemented block basic sequence $(\uu_m)_{m \in \Mt}$ for which $(\Ind^*_{\supp(\uu_m)})_{m \in \Mt}$ is a sequence of projecting functionals.
\end{lemma}

\begin{proof}Let $C_1$ be such that $\sum_{n \in \Nt} |\xx_n^*(f)|\le C_1 \Vert f \Vert$ for every $f\in\XX$. Set \[
C_2=\sup_{m \in \Mt} \Vert \uu_m\Vert,\quad C_3=\sup_{m \in \Mt} \Vert \uu_m^*\Vert,\;\quad \text{and}\;
C_4=\sup_{n \in \Nt} \Vert \xx_n\Vert.\]
Fix $m \in \Mt$ and put
\[
A_m=\left\{n \in \Nt \colon |\uu_m^*(\xx_n)| > \frac{1}{2C_1 C_2}\right\}.
\]
We have
\[
\sum_{n \in \Nt\setminus A_m} |\xx_n^*(\uu_m) \, \uu_m^*(\xx_n)| \le\frac{1}{2C_1 C_2} \sum_{n \in \Nt\setminus A_m} | \xx_n^*(\uu_m)|\le \frac{1}{2}.
\]
Hence,
\begin{align*}
\lambda_m:&=\sum_{n\in A_m} |\xx_n^*(\uu_m) \, \uu_m^*(\xx_n)|\\
&\ge -\frac{1}{2}+\sum_{n \in \Nt} |\xx_n^*(\uu_m) \, \uu_m^*(\xx_n)|\\
&\ge -\frac{1}{2}+\uu_m^*(\uu_m)=\frac{1}{2}.
\end{align*}
Let
\[
\vv_m =\lambda_m^{-1} \sum_{n\in A_m} |\xx_n^*(\uu_m) \, \uu_m^*(\xx_n)| \, \xx_n
\]
and $\vv_m^*=\Ind_{A_m}^*$. For every $n \in \Nt$ we have 
\[
\vv_m^*(\vv_m)=1, \quad \lambda_m^{-1} |\uu_m^*(\xx_n)|\le 2 C_3 C_4,
\] and for every $n\in A_m$,
\[
1\le 2 C_1 C_2 |\uu_m^*(\xx_n)|.
\] 
Hence, the result follows from \cite{AlbiacAnsorena2020}*{Lemma 3.1}.
\end{proof}

We will use the full force of the lattice structure induced by the basis in the following reduction lemma.

\begin{lemma}\label{lem:keytechniquebis}Let $\XX$ be a quasi-Banach space whose Banach envelope is anti-Euclidean. Suppose that $\BB$ is an L-convex, unconditional basis of $X$ which is universal for well complemented block basic sequences. Then, if $\BB_u$ is another unconditional basis of $\XX$, there are positive integers $s$ and $t$ such that $\BB_u$ is permutatively equivalent to a subbasis of $\BB^s$ and $\BB$ is permutatively equivalent to a subbasis of $\BB_u^t$.
\end{lemma}

\begin{proof} Since $\BB_u$ is lattice anti-Euclidean, \cite{AKL2004}*{Theorem 3.4} yields that $\BB_u$ is permutatively equivalent to a well complemented block basic sequence of $\BB^s$ for some $s\in\NN$. By \cite{AlbiacAnsorena2020}*{Proposition 3.4}, $\BB^s$ is universal for well complemented block basic sequences so that $\BB_u$ is permutatively equivalent to a subbasis of $\BB^s$. Since $\BB^s$ inherits the convexity from $\BB$, the basis $\BB_u$ is L-convex and universal for well complemented block basic sequences. Switching the roles of $\BB$ and $\BB_u$ yields the conclusion of the lemma.
\end{proof}

\begin{remark} A remark on the inherited order structure in a quasi-Banach lattice is in order here. Kalton showed in \cite{Kalton1984b}*{Theorem 4.2} that every unconditional basic sequence $\BB_0$ of a quasi-Banach space with an L-convex unconditional basis $\BB$ is L-convex. This argument would have, indeed, simplified the proof of Lemma~\ref{lem:keytechniquebis}. However, we wanted to make the point that the validity of the lemma does not depend on such a deep theorem as Kalton's.
\end{remark}

We are ready to prove the main result of this section.

\begin{theorem}\label{thm:keytechniquebis} Let $\XX$ be a quasi-Banach space whose Banach envelope is anti-Euclidean.
Suppose $\BB$ is an unconditional basis for $\XX$ such that:
\begin{enumerate}
\item[(i)] The lattice structure induced by $\BB$ in $\XX$ is L-convex;
\item[(ii)] $\BB$ is universal for well complemented block basic sequences; and
\item[(iii)] $\BB\sim \BB^2$.
\end{enumerate}
Then $\XX$ has a unique unconditional basis up to permutation.
\end{theorem}

\begin{proof} Let $\BB_u$ be another unconditional basis of $\XX$. Since $\BB^r \sim \BB$ for every $r\in\NN$, applying Lemma~\ref{lem:keytechniquebis} yields that $\BB_u$ is permutatively equivalent to a subbasis of $\BB$ and that $\BB^t$ is permutatively equivalent to a subbasis of $\BB_u^t$ for some $t\in\NN$. Combining Theorem~\ref{lem:SQRTSubases} with Theorem~\ref{thm:SBUB} yields $\BB_u\sim\BB$.
\end{proof}

Theorem~\ref{lem:SQRTSubases} becomes instrumental in reaching the conclusion of the previous theorem. Indeed, without it, and under the same hipotheses as in Theorem~\ref{thm:keytechniquebis}, we would have only been able to guarantee that given another unconditional basis $\BB_u$ of $\XX$, $\BB_u$ is permutatively equivalent to a subbasis of $\BB$ and that $\BB$ is permutatively equivalent to a subbasis of some $s$-fold product of $\BB_u$. Thanks to Theorem~\ref{lem:SQRTSubases} we can close close the ``gap" between $\BB$ and $\BB_u$ and arrive at the permutative equivalence of the two bases. Although this gap might seem small, we would like to emphasize that in the lack of Theorem~\ref{thm:keytechniquebis} the specialists were forced to use additional properties of $\BB$ to infer that $\BB$ is the unique unconditional basis of $\XX$. For instance, in the proof that $\ell_1(\ell_p)$, $0<p<1$, has a unique unconditional basis up to permutation, the authors used that all subbases of the canonical basis of $\ell_1(\ell_p)$ are permutatively equivalent to their square (see \cite{AKL2004}).

\section{Applicability of our scheme to anti-Euclidean spaces}\label{Sec:ApplAnti-Euclidean}

\noindent Most anti-Euclidean spaces scattered through the literature with a unique unconditional basis (up to permutation) fulfil the hypotheses of Theorem~\ref{thm:keytechniquebis}. This can be checked on by looking up the corresponding references contained herein. However, with the aim to be as self-contained as possible and for the convenience of the reader we next survey how to verify the hypotheses of Theorem~\ref{thm:keytechniquebis} in all known spaces (Banach and non-Banach) with a unique unconditional basis and some other new ones. The spaces in this section and the next will be the protagonists of Section~\ref{DirectSumsUnc}, where we will combine them to get the uniqueness of unconditional basis up to permutation of their finite direct sums.

In what follows, the symbol $\alpha_i\lesssim \beta_i$ for $i\in I$ means that the families of positive real numbers $(\alpha_i)_{i\in I}$ and $(\beta_i)_{i\in I}$ verify $\sup_{i\in I}\alpha_i/\beta_i <\infty$. If $\alpha_i\lesssim \beta_i$ and $\beta_i\lesssim \alpha_i$ for $i\in I$ we say $(\alpha_i)_{i\in I}$ are $(\beta_i)_{i\in I}$ are equivalent, and we write $\alpha_i\approx \beta_i$ for $i\in I$. 

\subsection{The space $\bm{\ell_{1}}$} The simplest example of an anti-Euclidean space is $\ell_1$. Since the canonical basis is perfectly homogeneous, it is universal for well complemented block basic sequences. Finally, since it is symmetric, 
it is equivalent to its square. 

\subsection{Orlicz sequence spaces}\label{ex:Orlicz}
An \textit{Orlicz function} will be a right-continuous increasing function $\varphi\colon[0,\infty)\to[0,\infty)$ such $\varphi(0)=0$, $\varphi(1)=1$ and $\varphi(s+t)\le C (\varphi(s) + \varphi(t))$ for some constant $C$ and every $s$, $t\ge 0$. The \textit{Orlicz space} $\ell_\varphi$ is the space associated to the Luxembourg quasi-norm defined from the modular $(a_n)_{n=1}^\infty \mapsto \sum_{n=1}^\infty \varphi(|a_n|)$. Our assumptions on $\varphi$ yield that $\ell_\varphi$ is a symmetric sequence space. Kalton proved in \cite{Kalton1977} that if $\varphi$ satisfies
\begin{equation}\label{eq:Orlicz1}
t\lesssim \varphi(t), \quad 0\le t \le 1,
\end{equation}
and
\begin{equation}\label{eq:OrliczK21}
\Lambda_\varphi:=\lim_{\varepsilon\to 0^+} \inf_{0<s<1}\frac{-1}{\log \varepsilon}\int_\varepsilon^1 \frac{\varphi(sx)}{sx^2}\, dx=\infty,
\end{equation}
then $\ell_\varphi$ has a unique unconditional basis up to permutation. It is easy to show that \eqref{eq:Orlicz1} implies that the Banach envelope of $\ell_\varphi$ is anti-Euclidean, and it is implicit in \cite{Kalton1977} that if \eqref{eq:Orlicz1} and \eqref{eq:OrliczK21} hold, then the unit vector system of $\ell_\varphi$ is universal for well complemented block basic sequences. For the sake of completeness and further reference, we record these results and sketch a proof of them.

\begin{proposition}[cf. \cite{Kalton1977}] Let $\varphi$ be an Orlicz function such that both \eqref{eq:Orlicz1} and \eqref{eq:OrliczK21} hold. Then:
\begin{itemize}
\item[(i)] The inclusion map from $\ell_\varphi$ in $\ell_1$ is an envelope map.
\item[(ii)] The unit vector system of $\ell_\varphi$ has the peaking property.
\end{itemize}
\end{proposition}

\begin{proof}Since $\ell_1$ is the Orlicz sequence space associated to the function $t\mapsto t$, we have $\ell_\varphi\subseteq\ell_1$. Then, (i) follows from Lemma~\ref{lem:BEl1}.

Assume by contradiction that $\BB_u=(\uu_m)_{m \in \Mt}$ is a a well complemented block basic sequence of $\ell_\varphi$, that $(\uu_m^*)_{m \in \Mt}$ is a family of well complemented projecting functionals for $\BB_u$, but that
\[
\inf_{m \in \Mt} \sup_{n\in\NN} |\uu_m^*(\ee_n)| \, |\ee_n^*(\uu_m)|=0.
\]
Then, by \cite{Kalton1977}*{Theorem 6.5}, $\ell_\varphi$ has a a complemented basic sequence $\BB_y$ such that $\YY=[\BB_y]$ is locally convex. Using (i) and \cite{AlbiacAnsorena2020}*{Lemma 2.1}, it follows that the restriction of the inclusion map of $\ell_\varphi$ in $\ell_1$ to $\YY$ is an isomorphism. Therefore, by \cite{Kalton1977}*{Theorem 5.3}, we reach the absurdity that $\Lambda_\varphi<\infty$.
\end{proof}

\subsection{Lorentz sequence spaces}\label{ex:Lorentz}
Let $\ww=(w_n)_{n=1}^\infty$ be a \textit{weight}, i.e., a sequence of positive scalars, and $0<p<\infty$. Suppose that $\ww$ decreases to zero. The \textit{Lorentz space} $d(\ww,p)$ is the quasi-Banach space consisting of all $f=(a_n)_{n=1}^\infty\in\FF^\NN$ such that
\[
\Vert f \Vert_{d(\ww,p)} =\sup_{\pi\in\Pi} \left(\sum_{n=1}^\infty | a_{\pi(n)}|^p \, w_n \right)^{1/p}<\infty,
\]
where $\Pi$ is the set of all permutations of $\NN$. The unit vector system is a symmetric basis of $d(\ww,p)$. It was proved in \cite{AlbiacLeranoz2008} that if the weight fulfils the condition
\begin{equation}\label{eq:lorentk2}
\inf_{k\in \NN} \frac{\displaystyle\sum_{n=1}^k w_n}{k^{p}}>0,
\end{equation}
then $d(\ww,p)$ has a unique unconditional basis up to permutation. Next, we deduce this result by combining Theorem~\ref{thm:keytechniquebis} with arguments from \cite{AlbiacLeranoz2008}.

\begin{proposition}[cf. \cite{AlbiacLeranoz2008}] Let $0<p<1$ and $\ww=(w_n)_{n=1}^\infty$ decreasing to zero. Then $d(\ww,p)\subseteq\ell_1$ if and only if \eqref{eq:lorentk2} holds. Moreover, if \eqref{eq:lorentk2} holds, then
\begin{itemize}
\item[(i)] the Banach envelope of $d(\ww,p)$ is $\ell_1$ via the inclusion map, and
\item[(ii)] the unit vector system of $d(\ww,p)$ has the peaking property.
\end{itemize}
\end{proposition}

\begin{proof}For $k\in\NN$ write $s_k=\sum_{n=1}^k w_n$. Assume that $d(\ww,p)\subseteq\ell_1$ and let $C$ be the norm of the inclusion map. If $|A|=k$ we have 
\[\Vert \Ind_A\Vert_1=k,\quad \text{and} \quad \Vert \Ind_A\Vert_{\ww,p}=s_k^{1/p}.\] Thus $k\le C s_k^{1/p}$ for every $k\in\NN$.

The weak-Lorentz space $d_\infty(\uu,p)$ associated to a weight $\uu=(u_n)_{n=1}^\infty$ and $0<p<\infty$ consists of all sequences $f\in c_0$ whose non-increasing rearrangement $(a_k^*)_{k=1}^\infty$ satisfies
\[
\Vert f \Vert_{d_\infty(\uu,p)} =\sup_k\left(\sum_{n=1}^k u_n\right)^{1/p} a_k^*<\infty.
\]
We have $d_\infty(\uu,p)\subseteq d(\uu,p)$ for every $0<p<\infty$ and every weight $\uu$.
If $\uu_p=(n^p-(n-1)^p)_{j=1}^\infty$ the rearrangement inequality and the mere definition of the spaces yields
\[
[d(\uu_p,p)]^p \cdot [d_\infty(\uu_p,p)]^{1-p} \subseteq \ell_1.
\]
We also have the obvious inclusion
\[
d(\uu_p,p) \subseteq [d(\uu_p,p)]^p \cdot [d(\uu_p,p)]^{1-p}.
\]
Summing up, we obtain $d(\uu_p,p)\subseteq\ell_1$.

Assume that $\ww$ fulfils \eqref{eq:lorentk2}. We deduce that $ d(\ww,p) \subseteq d(\uu_p,p)$. Therefore, $d(\ww,p) \subseteq \ell_1$. Then, (i) follows from Lemma~\ref{lem:BEl1}. To prove (ii), we pick a semi-normalized well complemented block basic sequence $(\uu_m)_{m \in \Mt}$ with good projecting functionals $(\uu_m^*)_{m \in \Mt}$. By Lemma~\ref{lem:k2two}, we can suppose that $\uu_m^*= \Ind^*_{\supp(\uu_m)}$ so that
\[
\sup_{n\in \NN} | \uu_m^*(\ee_n)| \, |\ee_n^*(\uu_m)|= \sup_{n\in \NN} |\ee_n^*(\uu_m)|.
\]
Finally, note that the proof of \cite{AlbiacLeranoz2008}*{Theorem 2.4} gives
\[
\inf_{m \in \Mt} \sup_{n\in \NN} |\ee_n^*(\uu_m)|>0.\qedhere
\]
\end{proof}

\subsection{Tsirelson's space} Casazza and Kalton established in \cite{CasKal1998} the uniqueness of unconditional basis up to permutation of Tsirelson's space $\Ts$ and its complemented subspaces with unconditional basis as a byproduct of their study of complemented basic sequences in lattice anti-Euclidean Banach spaces. Their result answered a question by Bourgain et al.\ (\cite{BCLT1985}), who had proved the uniqueness of unconditional basis up to permutation of the $2$-convexifyed Tsirelson's space $\Ts^{(2)}$ of $\Ts$ (see Example~\ref{pconvTsi} in \S~\ref{SectStrongAbs} for the definition). Unlike $\Ts^{(2)}$, which is ``highly" Euclidean, the space $\Ts$ is anti-Euclidean. To see the latter requires the notion of dominance, introduced in \cite{CasKal1998}. 

Let $\BB=(\xx_n)_{n=1}^\infty$ be a (semi-normalized) unconditional basis of a quasi-Banach space $\XX$. Given $f$, $g\in\XX$, we write $f\prec g$ if $m<n$ for all $m\in\supp(f)$ and $n\in\supp(g)$. The basis $\BB$ is said to be \textit{left (resp.\ right) dominant} if there is a constant $C$ such that whenever $(f_i)_{i=1}^N$ and $(g_i)_{i=1}^N$ are disjointly supported families with $f_i\prec g_i$ (resp. $g_i\prec f_i$) and $\Vert f_i\Vert \le \Vert g_i\Vert$ for all $i\in\NN[N]$, then $\Vert \sum_{i=1}^N f_i\Vert \le C \Vert \sum_{i=1}^N g_i\Vert$. If $\XX$ is a Banach space with a left (resp.\ right) dominant unconditional basis $\BB$ there is a unique $r=r(\BB)\in[1,\infty]$ such that $\ell_r$ is finitely block representable in $\XX$. In the case when $r(\BB)\in\{1,\infty\}$, $\XX$ is anti-Euclidean (see \cite{CasKal1998}*{Proposition 5.3}).

The canonical basis of the \textit{Tsirelson space} $\Ts$ is right dominant \cite{CasKal1998}*{Proposition 5.12}, and $r(\Ts)=1$. Moreover, by \cite{CasKal1998}*{Proposition 5.5} and \cite{CasShu1989}*{page 14}, the canonical basis (as well as each of its subases) is equivalent to its square. In our language, \cite{CasKal1998}*{Theorem 5.6} says that every left (resp.\ right) dominant unconditional basis is universal for well complemented block basic sequences. Finally, since it is locally convex, $\Ts$ is trivially an L-convex lattice. 

\subsection{Bourgin-Nakano spaces.} Let $\Nt$ be a countable set. A \textit{Bourgin-Nakano index} is a family $(p_n)_{n \in \Nt}$ in $(0,\infty)$ with $p=\inf_n p_n>0$. The \textit{Bourgin-Nakano space} $\ell(p_{n})$ is the quasi-Banach space built from the modular
\[
m_{(p_n)}\colon\FF^{\Nt} \to[0,\infty), \quad (a_n)_{n \in \Nt} \mapsto \sum_{n \in \Nt} |a_n|^{p_n}.
\]
Note that, by the Monotone Convergence Theorem, the closed unit ball of $\ell(p_{n})$ is the set
\[
B_{\ell(p_{n})}=\{ f \in\FF^{\Nt} \colon m_{(p_n)}(f)\le 1\}.
\]
If we endow $\ell(p_{n})$ with the natural ordering, it becomes a $p$-convex quasi-Banach lattice. The separable part $h(p_{n})=[\ee_n \colon n \in \Nt]$ of $\ell(p_{n})$ is a sequence space. We have $\ell(p_{n})=h(p_{n})$ if and only if $\sup_n p_n<\infty$.

These spaces where introduced by Bourgin \cite{Bourgin1943} in the particular case when $p_n\le 1$ for every $n\in\Nt$. Nakano \cite{Nakano1950} studied the case when $p_n\ge 1$ for every $n \in \Nt$, so that the arising spaces are locally convex, i.e., Banach spaces.

Let us record some results on Bourgin-Nakano spaces of interest for the purposes of this paper.
\begin{lemma}\label{lem:dominancyBN}Let $(p_{n})_{n \in \Nt}$ and $(q_m)_{m \in \Mt}$ be Bourgin-Nakano indexes. Let $\BB_u=(\uu_j)_{j=1}^\infty$ and $\BB_v=(\vv_j)_{j=1}^\infty$ be normalized block basic sequences in $\ell(p_{n})$ and $\ell(q_{n})$ respectively. Suppose that $p_n \le q_m$ for all $(n,m)\in\supp(\uu_j)\times\supp(\vv_j)$ and all $j\in\NN$. Then $\BB_u$ $1$-dominates $\BB_v$.
\end{lemma}

\begin{proof}Let $j\in \NN$. Pick $r_j\in[1,\infty)$ such that $p_n\le r \le q_m$ for every $n\in A_j:=\supp(\uu_j)$ and $m\in B_j:=\supp(\vv_j)$. Put $\uu_j=\sum_{n\in A_j} a_j\, \ee_j$ and $\vv_j=\sum_{n\in A_j} b_j\, \ee_j$. Since $\Vert\uu_j\Vert=\Vert \vv_j\Vert=1$, we have
\[
\sum_{n\in A_j} |a_j|^{p_n}=1=\sum_{m\in B_j} |b_m|^{q_m}.
\]

Let $f=\sum_{j=1}^\infty c_j \, \uu_j\in B_{\ell(p_n)}$. We have $|c_j|\le 1$ for every $j\in\NN$. Hence,
\begin{align*}
m_{(q_n)}\left(\sum_{j=1}^\infty c_j \, \vv_j\right)&=\sum_{j=1}^\infty \sum_{m\in B_j} |c_j|^{q_m} |b_m|^{q_m}\\
&\le \sum_{j=1}^\infty |c_j|^{r} \sum_{m\in B_j} |b_m|^{q_m}\\
&=\sum_{j=1}^\infty |c_j|^{r} \sum_{n\in A_j} |a_n|^{p_n}\\
&\le \sum_{j=1}^\infty \sum_{n\in A_j} |c_j|^{p_n} |a_n|^{p_n}\\
&\le 1.
\end{align*}
Therefore, $\sum_{j=1}^\infty c_j \, \uu_j\in B_{\ell(q_n)}$.
\end{proof}

\begin{proposition}[see \cite{CasKal1998}*{Proof of Theorem 5.8}]\label{prop:BNAE}Let $(p_n)_{n=1}^\infty$ be a non-increasing (resp. non-decreasing) Bourgin-Nakano index. Then, the unit vector system of $\ell(p_{n})$ is right (resp. left) dominant. Moreover, $r(\ell(p_{n}))=\lim_n p_n$.
\end{proposition}
\begin{proof}It is a consequence of Lemma~\ref{lem:dominancyBN}.
\end{proof}
Given $(p_n)_{n \in \Nt}$ we put $(\widehat{p_{n}})_{n\in \Nt}= (\max\{1,p_n\})_{n \in \Nt}$.
\begin{proposition}\label{prop:BNEnv}Let $(p_n)_{n \in \Nt}$ be a Bourgin-Nakano index. Then the Banach envelope of $\ell(p_{n})$ is $\ell(\widehat{p_{n}})$ via the inclusion map.
\end{proposition}

\begin{proof} Put $\Nt_{b}=\{ n \in \Nt \colon p_n<1\}$, $\Nt_{k}=\{n \in \Nt \colon p_n\ge 1\}$. The obvious map from $\FF^{\Nt}$ onto
$\FF^{\Nt_b} \times \FF^{\Nt_k}$ restricts to a lattice isomorphism from $\ell(p_{n})$ onto $\ell(p_{n})_{n\in \Nt_{b}} \oplus \ell(p_{n})_{n\in \Nt_{k}}$. Hence, by \cite{AlbiacAnsorena2020}*{Lemma 2.3}, we can assume without loss of generality that $\Nt_k=\emptyset$. In this particular case, since $\sum_{n \in \Nt} |a_n|\le 1$ for every $(a_n)_{n \in \Nt} \in B_{\ell(p_{n})}$ and $\ee_n\in B_{\ell(p_{n})}$ for every $n \in \Nt$, the closed convex hull of $B_{\ell(p_{n})}$ in $\ell_1(\Nt)$ is the unit closed ball of $\ell_1(\Nt)$. Since $\ell(\widehat{p_{n}})=\ell_1(\Nt)$ isometrically, we infer that the Banach envelope of $\ell(p_{n})$ is $\ell(\widehat{p_{n}})$ isometrically under the inclusion map.
\end{proof}

\begin{corollary}\label{cor:BNAE}Let $(p_n)_{n \in \Nt}$ be a Bourgin-Nakano index. Suppose that $\limsup_n p_n \le 1$. Then, the Banach envelope of $\ell(p_{n})$ is anti-Euclidean.
\end{corollary}

\begin{proof}Just combine Propositions~\ref{prop:BNAE} and \ref{prop:BNEnv}.
\end{proof}

\begin{proposition}\label{prop:BNUWC}Let $(p_n)_{n=1}^\infty$ be a Bourgin-Nakano index. Then, the unit vector system of $\ell(p_{n})$ is universal for well complemented block basic sequences.
\end{proposition}

\begin{proof}
Let $\BB_y=(\yy_m)_{m \in \Mt}$ be a semi-normalized well complemented block basic sequence and let $(\uu_m^*)_{m \in \Mt}$ be a sequence of good projecting functionals. Since
\[
\sum_{n \in \Nt} \ee_n^*(\yy_m) \, \yy_m^*(\ee_n)=\yy_m^*(\yy_m)=1
\]
for every $m \in \Mt$, there are families $(A_m)_{m \in \Mt}$ and $(B_m)_{m \in \Mt}$ of subsets of $N$ and $\pi\colon \Mt \to \Nt$ such that, if
\[
\lambda_m=\sum_{n\in A_m} \ee_n^*(\yy_m) \, \yy_m^*(\ee_n) \text{ and } \mu_m=\sum_{n\in B_m} \ee_n^*(\yy_m) \, \yy_m^*(\ee_n),
\]
then $\min\{| \lambda_m|,|\mu_m|\}\ge 1/2$, $A_m\cup B_m=\supp(\yy_m)$, $A_m\cap B_m=\{\pi(m)\}$, and
\begin{equation*}
\max_{n\in A_m} p_n = \min_{n\in B_m} p_n =p_{\pi(m)}
\end{equation*}
for every $m \in \Mt$. Let $\uu_m=S_{A_m}(\yy_m)$, $\uu_m^*=S_{A_m}^*(\yy_m^*)$, $\vv_m=S_{B_m}(\yy_m)$, and $\vv_m^*=S_{B_m}^*(\yy_m)$ for $m \in \Mt$. Since $\uu_m^*(\uu_m)=\lambda_m$ and $\vv_m^*(\vv_m)=\mu_m$ for every $m \in \Mt$, applying \cite{AlbiacAnsorena2020}*{Lemma 3.1} yields that both $\BB_u=(\uu_m)_{m \in \Mt}$ and $\BB_v=(\vv_m)_{m \in \Mt}$ are well complemented block basic sequences equivalent to $\BB_y$. By Lemma~\ref{lem:dominancyBN}, $\BB_u$ dominates $\BB:=(\ee_{\pi(m)})_{m \in \Mt}$ and, in turn, $\BB$ dominates $\BB_v$. We infer that $\BB_y$ and $\BB$ are equivalent.
\end{proof}

\begin{proposition}\label{prop:BNSQ}Let $(p_n)_{n \in \Nt}$ be a Bourgin-Nakano index. The unit vector system of $\ell(p_{n})$ is equivalent to its square if and only if there is a partition $(\Nt_1,\Nt_2)$ of $\Nt$ and bijections $\pi_i\colon \Nt \to \Nt_i$, $i=1$, $2$,
such that, for some $0<c<1$,
\begin{equation*}
\sum_{n \in \Nt} c^{\frac{ p_n q_{i,n} }{ |p_n - q_{i,n}| }}<\infty, \quad i=1,2,
\end{equation*}
where if $q_{i,n} = p_{\pi_i(n)}$.
\end{proposition}

\begin{proof}This result follows from \cite{Nakano1951}*{Theorem 1}, which characterizes when two (a priori different) Bourgin-Nakano spaces are identical.
\end{proof}

We remark that, in certain cases, we can give a more simple characterization of Nakano spaces that are lattice isomorphic to its square. For instance, if $(p_n)_{n=1}^\infty$ is a monotone sequence, then $\ell(p_{n})$ is lattice isomorphic to its square if and only if
\[
\left| \frac{1}{p_n} - \frac{1}{p_{2n}} \right|\lesssim \frac{1}{1+\log(n)}, \quad n\in\NN
\]
(see \cite{CasKal1998}*{Proof of Theorem 5.8}).
\begin{theorem}\label{Nakanouncthm} Suppose that Bourgin-Nakano index $(p_n)_{n \in \Nt}$  satisfies $\limsup_n p_n \le 1$. Suppose also that there exist a partition $(\Nt_1,\Nt_2)$ of $\Nt$, and bijections $\pi_i\colon \Nt \to \Nt_i$, $i=1$, $2$, so that
\begin{equation*}
\sum_{n \in \Nt} c^{1/ | { p_n } -{ p_{\pi_i(n)}}|}<\infty, \quad i=1,2,
\end{equation*}
for some $0<c<1$. 
Then $\ell(p_{n})$ has a unique unconditional basis up to permutation.
\end{theorem}
\begin{proof}Just combine Corollary~\ref{cor:BNAE}, Proposition~\ref{prop:BNUWC}, Proposition~\ref{prop:BNSQ} and Theorem~\ref{thm:keytechniquebis}.
\end{proof}

An important class of anti-Euclidean spaces arises from a special type of bases called strongly absolute. We tackle this case separately in the next section.

\section{Applicability to spaces with strongly absolute bases}\label{SectStrongAbs} 

\noindent In the category of bases one could say that strongly absolute bases are ``purely nonlocally convex'' bases, in the sense that if a quasi-Banach space $X$ has a strongly absolute basis, then its unit ball is far from being a convex set and so $X$ is far from being a Banach space. The term strongly absolute for a basis was coined in \cite{KLW1990}. Here we give a slightly different, but equivalent, definition. We say that a (semi-normalized) unconditional basis $\BB=(\xx_n)_{n \in \Nt}$ of a quasi-Banach space $\XX$ is \textit{strongly absolute} if for every $\varepsilon>0$ there is a constant $0<C(\varepsilon)$ such that
\begin{equation}\label{eq:sb}
\sum_{n \in \Nt} |\xx_n^*(f)| \le \max\left\{ C(\varepsilon) \sup_{n \in \Nt} |\xx_n^*(f)| , \varepsilon \Vert f \Vert \right\}, \quad f\in\XX.
\end{equation}

In the following lemma we record a key property of strongly absolute bases. The proof is straightforward and so we omit it. 
\begin{lemma}\label{lem:k2three}Let $\BB=(\xx_n)_{n \in \Nt}$ be a strongly absolute unconditional basis of a quasi-Banach space $\XX$. Suppose that $V\subseteq\XX$ is such that
$\inf_{f\in V}{\Vert f \Vert^{-1}} \Vert \Fou(f)\Vert_1 >0$. Then, $\inf_{f\in V}{\Vert f \Vert^{-1}} \Vert \Fou(f)\Vert_\infty>0$.
\end{lemma}

\begin{proposition}[cf.\ \cite{KLW1990}]\label{prop:K2SA}Let $\BB$ be a strongly absolute unconditional basis of a quasi-Banach space $\XX$. Then:
\begin{itemize}
\item[(i)] The Banach envelope of $\XX$ is $\ell_1$ via the coefficient transform.
\item[(ii)] $\BB$ has the peaking property.
\end{itemize}
\end{proposition}

\begin{proof}It is clear that $\BB$ dominates the unit vector system of $\ell_1$, so that (i) follows from Lemma~\ref{lem:BEl1}.

Let $\BB_u=(\uu_m)_{m \in \Mt}$ be a semi-normalized well complemented block basic sequence. By Lemma~\ref{lem:k2two} we may assume that $(\uu_m^*)_{m \in \Mt}=(\Ind^*_{\supp(\uu_m)})_{m \in \Mt}$ is a sequence of good projecting functionals for $\BB_u$. Using (i) and \cite{AlbiacAnsorena2020}*{Lemma 2.1} we deduce that the sequence $(\Fou(\uu_m))_{m=1}^\infty$ is semi-normalized in $\ell_1$. Therefore, 
\[
\inf_m \Vert \uu_m\Vert^{-1} \Vert \Fou(\uu_m)\Vert_1>0.
\]
Lemma~\ref{lem:k2three} yields
\begin{align*}
\inf_{m \in \Mt} \sup_{n \in \Nt} |\uu_m^*(\xx_n)| \, |\xx_n^*(\uu_m)|
&= \inf_{m \in \Mt} \Vert \Fou(\uu_m)\Vert_\infty\\
&\ge \inf_{m \in \Mt} \Vert \uu_m \Vert \inf_{m \in \Mt} \frac{\Vert \Fou(\uu_m)\Vert_\infty}{\Vert \uu_m \Vert}
>0.\qedhere
\end{align*}
\end{proof}

Combining Proposition~\ref{prop:K2SA} with Theorem~\ref{thm:keytechniquebis} immediately yields the following general result.

\begin{corollary}\label{weakWoj} Let $\XX$ be a quasi-Banach space with a strongly absolute unconditional basis which induces an L-convex structure on $X$. If $\BB$ is equivalent to its square, then $\XX$ has a unique unconditional basis up to permutation.
\end{corollary} 

Wojtaszczyk obtained in \cite{Woj1997} the uniqueness of unconditional basis of a quasi-Banach space $\XX$ under the same hypotheses as in Corollary~\ref{weakWoj} replacing $\BB^2\sim \BB$ with the weaker assumption that $\XX^s\approx\XX$ for some $s\ge 2$. For the sake of completeness, we next show how we can combine the techniques from \cite{Woj1997} to pass from the condition ``$\XX^s\approx\XX$ for some $s\ge 2$'' to ``$\BB^2\sim \BB$".

\begin{theorem}[cf.\ \cite{Woj1997}*{Theorem 2.12}]\label{thm:WoUltimate} Let $X$ be a quasi-Banach space with a strongly absolute unconditional basis $\BB$ that induces an $L$-convex lattice structure on $X$. If $\XX^s\approx\XX$ for some $s\ge 2$ then $\BB^2\sim \BB$; in particular $\XX^2\approx\XX$.
\end{theorem}

\begin{proof}Put $\BB^s=(\yy_m)_{m \in \Mt}$. We have that $\BB^{s^2}=(\yy_{i,m})_{(i,m)\in\NN[s]\times \Mt}$ is permutatively equivalent to a basis of $\XX\approx\XX^{s^2}$. Hence, by \cite{Woj1997}*{Proposition 2.10}, there is $\alpha\colon \Mt\to \NN[s]$ such that $\BB'=(\yy_{\alpha(m),m})_{m \in \Mt}$ is permutatively equivalent to a subbasis of $\BB$. By Lemma~\ref{lem:1}, $\BB^s$ is equivalent to $\BB'$. Since $\BB$ is permutatively equivalent to a subbasis of $\BB^2$ and $\BB^2$ is permutatively equivalent to a subbasis of $\BB^s$, applying Theorem~\ref{thm:SBUB} yields $\BB^s\sim\BB^2\sim\BB$.
\end{proof}

As we said before, a strongly absolute unconditional basis can be thought of as a basis that dominates the canonical basis of $\ell_1$ but it is far from it. This intuition is substantiated by the following elementary result whose proof we omit.

\begin{lemma}\label{lem:SADom}Let $\BB_x$ and $\BB_y$ be unconditional bases of quasi-Banach spaces $\XX$ and $\YY$ respectively. Suppose that $\BB_x$ dominates $\BB_y$ and that $\BB_y$ is strongly absolute. Then $\BB_y$ is strongly absolute.
\end{lemma}

To complement the theoretical contents of this section we shall introduce a quantitative tool from approximation theory that measures how far an unconditional basis is from the canonical $\ell_{1}$-basis. 

Given an unconditional basis $\BB$ of a quasi-Banach space $X$, its \textit{lower democracy function} is defined as
\[
\varphi_m^l[\BB]=\inf_{|A|\ge m} \Vert \Ind_A[\BB]\Vert, \quad m\in\NN.
\]
Note that if $\BB$ is strongly absolute then
\[
\lim\limits_{m\to \infty} \frac{1}{m} \varphi_m^l[\BB]=\infty.
\]
The following result establishes that, conversely, if $(\varphi_m^l[\BB])_{m=1}^\infty$ is sufficiently far away from the sequence $(m)_{m=1}^\infty$, then the basis $\BB$ is strongly absolute.

\begin{proposition}\label{prop:SADem}Let $\BB=(\xx_n)_{n=1}^\infty$ be an unconditional basis of a quasi-Banach space $\XX$. Suppose that there exists $0<p<1$ such that for some constant $0<C$ we have
\[
m^{1/p} \le C \varphi_m^l[\BB],\quad m\in\NN.
\]
Then $\BB$ is strongly absolute.
\end{proposition}

\begin{proof} We may regard $X$ as a sequence space whose basis $\BB$ is just the unit vector system.
Pick $r\in(p,1)$. By \cite{AADK2016}*{Lemma 6.1}(a),
\[
\XX\subseteq \ell_{p,\infty} \subseteq \ell_r
\]
continuously. Since the canonical basis of $\ell_r$ is strongly absolute (see \cite{Leranoz1992}*{Lemma 2.2}), by Lemma~\ref{lem:SADom} the proof is over.
\end{proof}

We will use Proposition~\ref{prop:SADem} to readily deduce that the following important examples of bases, which are permutatively equivalent to their square, are strongly absolute. 

\begin{example} Given $0<p_i<1$ for $i\in\NN[n]$, the canonical basis of the mixed norm space $\ell_{p_1}(\cdots\ell_{p_i}(\cdots(\ell_{p_n})))$ is unconditional, strongly absolute, and induces a structure of L-convex lattice on the whole space. \end{example} 

\begin{example}Let $d\in\NN$. The canonical basis $\BB$ of the Hardy spaces $H_p(\TT^d)$, $0<p<1$ (see \cite{KLW1990}) satisfies
\[
m^{1/p} \approx \varphi_m^l[\BB, H_p(\TT^d)],\quad m\in\NN.
\]
Hence, $\BB$ is strongly absolute.
\end{example}

\begin{example}\label{TribelLizEx} Given a dimension $d\in\NN$, let $\Theta_d=\{0,1\}^d \setminus\{ 0\}$ and consider the set of indices 
\[
\Lambda_d= \ZZ \times \ZZ^d \times \Theta_d.
\]
The homogeneous Triebel-Lizorkin sequence space $\ring{\tl}_{p,q}^{s,d}$ of indeces and $p\in (0,\infty)$ and $q\in (0,\infty]$ and smoothness $s\in\RR$ consists of all scalar sequences $f=(a_\lambda)_{\lambda\in\Lambda}$ for which
\[
\Vert f \Vert_{\tl_{p,q}^{s}}= \left\Vert \left(\sum_{j=-\infty}^\infty \sum_{\delta\in\Theta_d} \sum_{n\in\ZZ^d} 2^{jq(s+d/2)} |a_{j,n,\delta}|^q \chi_{Q(j,n)}\right)^{1/q}\right\Vert_p<\infty,
\]
were $Q(j,n)$ denotes the cube of length $2^{-j}$ whose lower vertex is $2^{-j}n$. If we restrict ourselves to non-negative ``levels'' $j$ and we add $\ell_p$ as a component we obtain the inhomegeneous Triebel-Lizorkin sequence spaces. To be precise, set
\[
\Lambda_d^+=\{(j,n,\delta)\in\Lambda_d \colon j\ge 0\},
\]
and define
\[
\tl_{p,q}^{s,d} =\ell_p(\ZZ^d) \oplus \{ f=(a_\lambda)_{\lambda\in\Lambda_d^+} \colon \Vert f \Vert_{\tl_{p,q}^{s}}<\infty \}.
\]
It is known that the wavelet transforms associated to certain wavelet bases normalized in the $L_{2}$-norm are isomorphisms from $F_{p,q}^s(\RR^d)$ (respectively $\ring{F}_{p,q}^s(\RR^d)$ onto $\tl_{p,q}^s(\RR^d)$ (resp., $\ring{\tl}_{p,q}^{s,d}$). See \cite{FrJaWe1991}*{Theorem 7.20} for the homegeneous case and \cite{TriebelIII2006}*{Theorem 3.5} for the inhomogenous case. Thus, Triebel-Lizorkin spaces are isomorphic to the corresponding sequence spaces, and the aforementioned wavelet bases (regarded as distributions on Triebel-Lizorkin spaces) are equivalent to the unit vector systems of the corresponding sequence spaces. 

A similar technique to the one used by Temlyakov in \cite{Temlyakov1998} to prove that the Haar system is a democratic basis for $L_p$ when $1<p<\infty$ allows us to prove that the unique vector system $\EE$ of $\ring{\tl}_{p,q}^{s,d}$ satisfies
\[
m^{1/p} \approx \varphi_m^l[\EE, \ring{\tl}_{p,q}^{s,d}],\quad m\in\NN.
\]
Consequently, if $p<1$, the unique vector system of both $\ring{\tl}_{p,q}^{s,d}$ and $\tl_{p,q}^{s,d}$ is a strongly absolute unconditional basis. \end{example} 

\begin{example}\label{pconvTsi} Given $0<p<\infty$, the \textit{$p$-convexified Tsirelson's space}, denoted ${\Ts}^{(p)}$, is obtained from $\Ts$ by putting 
\begin{equation}\label{Tsirelsonnorm}
\Vert x\Vert_{{\Ts}^{(p)}} = \Vert (|a_{n}|^{p})_{n=1}^{\infty}\Vert_{\Ts}^{1/p}
\end{equation}
for those sequences of real numbers $x= (a_{n})_{n=1}^{\infty}$ such that $(|a_{n}|^{p})_{n=1}^{\infty}\in {\Ts}$. 
Equation~\eqref{Tsirelsonnorm} defines a norm for $1\le p$ and a $p$-norm when $0<p<1$. Obviously, the space $({\Ts}^{(1)}, \Vert \cdot\Vert_{{\Ts}^{(1)}})$ is simply $({\Ts}, \Vert \cdot\Vert_{\Ts})$.

For $0<p<\infty$, the canonical basis $\EE$ of ${\Ts}^{(p)}$  is $1$-unconditional,  it is permutatively equivalent to its square, and satisfies
\[
m^{1/p} \approx \varphi_m^l[\EE, {\Ts}^{(p)}],\quad m\in\NN.
\]
Hence in particular if $0<p<1$, $\EE$ is strongly absolute.
\end{example} 

\section{Uniqueness of unconditional basis of sums of anti-Euclidean spaces}\label{DirectSumsUnc}

\noindent Our last application of Theorem~\ref{thm:keytechniquebis} establishes that the uniqueness of unconditional bases up to permutation of anti-Euclidean quasi-Banach spaces is preserved by \textit{finite} direct sums.

\begin{theorem}\label{thm:DirecSums} Let $(\XX_i)_{i\in F}$ be a finite family of quasi-Banach spaces whose Banach envelopes are anti-Euclidean. Suppose that for each $i\in F$, 
$\BB_i$ is an unconditional basis of $\XX_{i}$ such that
\begin{enumerate}
\item[(i)] The lattice structure induced by $\BB_{i}$ in $\XX_{i}$ is L-convex;
\item[(ii)] $\BB_{i}$ is universal for well complemented block basic sequences; and
\item[(iii)] $\BB_{i}\sim \BB_{i}^2$.
\end{enumerate}
Then the space $\bigoplus_{i\in F} \XX_i$ has a unique unconditional basis up to permutation.
\end{theorem}

\begin{proof} Combining \cite{CasKal1999}*{Proposition 2.4} and \cite{AlbiacAnsorena2020}*{Lemma 2.3} we see that the Banach envelope of $\XX=\bigoplus_{i\in F} \XX_i$ is anti-Euclidean. It is clear that the basis $\BB=\bigoplus_{i\in F} \BB_i$ is L-convex and permutatively equivalent to its square. By \cite{AlbiacAnsorena2020}*{Proposition 3.4}, $\BB$ is universal for well complemented block basic sequences. So, the result follows from Theorem~\ref{thm:keytechniquebis}.
\end{proof}

Merging the results from Sections~\ref{Sec:ApplAnti-Euclidean} and \ref{SectStrongAbs} with Theorem~\ref{thm:DirecSums} provides new additions to the list of spaces with unique unconditional basis up to a permutation.

\begin{corollary} Let $F$ be a finite set of indeces. Suppose that for each $i\in F$, $X_{i}$ is one of the following spaces:
\begin{itemize}
\item[(i)] $\ell_{\varphi}$, where $\varphi$ verifies \eqref{eq:Orlicz1} and \eqref{eq:OrliczK21}, in particular $\ell_{p}$ for $p\le 1$;
\item[(ii)] $d(\ww, p)$, where $\ww$ verifies \eqref{eq:lorentk2};
\item[(iii)] $\Ts$;
\item[(iv)] $\ell(p_{n})$, where $(p_{n})_{n=1}^{\infty}$ verifies the hypothesis of Theorem~\ref{Nakanouncthm};
\item[(v)] $\ell_{p_1}(\cdots\ell_{p_i}(\cdots(\ell_{p_n})))$, where $0<p_i<1$ for $i\in\NN[n]$;
\item[(vi)] $H_p(\TT^{d})$ for $d\in \NN$ and $0<p<1$; 
\item[(vii)] $\ring{\tl}_{p,q}^{s,d}$ or $\tl_{p,q}^{s,d}$ as in Example~\ref{TribelLizEx};
\item[(viii)] ${\Ts}^{(p)}$ for $0<p<1$.
\end{itemize}
Then $\XX=\bigoplus_{i\in F} \XX_i$ has a unique unconditional basis up to permutation.

\end{corollary}



\end{document}